%% file: arxiv.tex
\documentclass[12pt]{article}

\usepackage{amsthm}
\include{includes}

\usepackage{graphics}

\begin{document}

\title{Reachability and Recurrence in a Modular Generalization of Annihilating Random Walks (and lights-out games) to hypergraphs}

\author{Gabriel Istrate}\thanks{ Department of Computer Science, West University of Timi\c{s}oara and -Austria Research Institute, Bd. V. P\^{a}rvan 4, cam. 045
B, Timi\c{s}oara, RO-300223, Romania. email: {\tt gabrielistrate@acm.org }}

\begin{abstract}
We study a discrete asynchronous dynamical system on hypergraphs that can be regarded as a natural extension of annihilating walks along two directions: first, the interaction topology is a hypergraph; second, the 
``number of particles`` at a vertex of the hypergraph is an element of a finite ring ${\bf Z}_{p}$ of integers modulo an odd  number  
$p\geq 3$. Equivalently particles move on a hypergraph, with a moving particle at a vertex being replaced by one indistinguishable copy at each neighbor in a 
given hyperedge; particles at a vertex collectively annihilate when their number reaches $p$.  
 
The boolean version of this system arose in earlier work \cite{istrate-balance} motivated by %earlier work % on 
 the statistical physics of social balance \cite{redner-balance,antal2005dynamics}, generalizes certain lights-out games \cite{sutner-lightsout} 
to finite fields and also has applications to the complexity of local search procedures for satisfiability. 
 
Our result shows that under a liberal sufficient condition on the nature of the interaction hypergraph 
there exists a polynomial time algorithm (based on linear algebra over ${\bf Z}_{p}$) for deciding reachability and recurrence of this 
dynamical system. Interestingly, we provide a counterexample that shows that this connection does {\em not} extend to all graphs.

\end{abstract}

\maketitle
\section{Introduction}	%) A SECTION HEADING

Interacting particle systems \cite{liggett-ips} are discrete dynamical systems, naturally related to cellular
 automata \cite{hanson1997computational,gacs2001reliable}, that have seen extensive study in the Statistical Physics of Complex Systems. While they are 
most naturally studied on lattices, extensions to general graphs are possible. Such extensions have recently found many applications to social dynamics, 
particularly as  {\em opinion formation models} (see \cite{castellano2009statistical} for a recent survey). In particular, the most popular interacting
particle systems,  the voter and antivoter model and (by duality) annihilating and coalescing random walks have also been studied on a general graph 
\cite{donnelly-welsh-interacting, donnelly-welsh-coloring,aldous-fill-book}. 

Extensions to hypergraphs are also possible and relevant in a social context. For instance, Lanchier and Neufer \cite{lanchier2012stochastic} argue 
for the naturalness of such an extension and give a spatial version of Galam's majority model \cite{galam2002minority} via a majority voting rule. 
Motivated by behavioral voting experiments on networks \cite{kearns2009behavioral}, Chung and Tsiatas study \cite{chung2012hypergraph} a voter model on 
hypergraphs. A final example comes from the Statistical Physics of social balance \cite{redner-balance,antal2005dynamics}. A dynamical adjustment process 
introduced in these papers naturally leads via duality
\cite{istrate-balance} to an extension of annihilating random walks to hypergraphs. This extension can be specified as follows: 

\begin{definition}{\bf [ANNIHILATING RANDOM WALKS ON HYPERGRAPHS:]}\\

Particles live on the vertices of a hypergraph. At each moment: 
\begin{enumerate}
\item We chose a random vertex $v$ containing a particle.  
\item We chose a random hyperedge $e$ that contains vertex $v$.  
\item Vertices in $e$ that contain a particle (including $v$) become empty. On the other hand empty vertices in $e$ will afterwards contain a particle. 
\end{enumerate} 
\label{arwhyp}
\end{definition}

The process specified at Step 3 can be described intuitively in the following way: the particle $P$ at vertex $e$ spawns 
a number of descendents, one for each vertex $z\in e\setminus v$, then dies. The new particles may meet already a pre-existing particle at
 vertex $z$, in which case the two particles "collectively annihilate".  This dynamics, studied in \cite{istrate-balance}, is also naturally 
related as we found out after completing \cite{istrate-balance}, to a classical problem in the area of combinatorial games, the theory of 
{\em lights-out games} \cite{sutner-lightsout}. 
We further discuss this connection in the next section. 

The remarkable aspect of the extension~(\ref{arwhyp}) of annihilating random walks to hypergraphs lies in its "explosive" nature: on hyperedges one particle may give birth to more than one copy. Thus, unlike the graph case, the total number of particles is generally {\em not} a nonincreasing function. 

The purpose of this paper is to study a modulo-$p$ version of the dynamical system from \cite{istrate-balance}, specifically, the following system:

\begin{definition}\label{maindef} 
 Let $p\geq 2$ be an integer. 
 A {\em ${\bf Z}_{p}$-annihilating walk on a hypergraph $G$} is defined as follows: each node $v$ of $G$ is initially 
endowed with a number $w(v)\in {\bf Z}_{p}$ (interpreted as number of {\em particles}). 

The allowed moves are specified as follows: choose a node $v$ such that $w(v)\neq 0$ and a hyperedge $e$ containing $v$. Change the state of $w(v)$ 
to $w(v)-1$. Also change the state of every node $u\neq v$, $u\in e$ to $w(u)+1$. 
\end{definition}

In other words: a number of indistinguishable particles  are initially placed 
at the vertices of $G$, each vertex holding from 0 to $p-1$ particles. At each step we choose a vertex $v$ containing at least one particle and a hyperedge 
containing $v$. We delete one particle at $v$ and add one particle at every vertex $w\neq v\in e$. If the number of 
particles at some $w$ reaches $p$, these $p$ particles are removed from $w$ (they ''collectively annihilate``). 

We are mainly interested in the complexity of the following two problems: 

\begin{definition}[\bf REACHABILITY]

Given hypergraph $G=(E,V)$ and states $w_{1},w_{2}\in {\bf Z}_{p}^{V}$, decide whether $w_{2}$ is reachable from $w_{1}$.  
\end{definition}

\begin{definition}[\bf RECURRENCE]

Given hypergraph $G=(E,V)$ and states $w_{1},w_{2}\in {\bf Z}_{p}^{V}$, decide whether $w_{2}$ is reachable from any state 
$w_{3}\in {\bf Z}_{p}^{V}$ 
reachable from $w_{1}$.   
\end{definition}

Of course, reachability and recurrence are fundamental prerequisites for studying the {\em random} version of this dynamical system as a finite-state 
Markov chains, the problem that was the original motivation of our research. 

There are simple algorithms that put the complexity of these two problems above in the complexity classes PSPACE and EXPSPACE, respectively: 
for REACHABILITY we simply consider reachability in the (exponentially large) state space directed graph $S$ with 
vertex set ${\bf Z}_{p}^{V}$. For RECURRENCE we combine 
enumeration of all vertices $w_{3}$ reachable from $w_{1}$ (via breadth first search) with testing reachability of $w_{2}$ from 
$w_{3}$. 

The main purpose of this paper is to show that under a quite liberal sufficient condition on the nature of underlying hypergraph reachability and recurrence questions 
for ${\bf Z}_{p}$-annihilating walks can be decided in polynomial time (actually they belong to 
the apparently weaker class $Mod_{p}$-L \cite{bdhm-mod}, but we won't discuss this issue here any further), by solving a certain 
system of linear equations over $Z_{p}$. 

Of course, the above result is not entirely surprising, as it comes in an established line of applications of linear algebra to 
reachability problems in lights-out games (see \cite{scherphuis-lightsout} for a discussion and list of references). On the other hand, as discussed in the 
next section, the class of moves we allow is more restricted than that in the models in \cite{scherphuis-lightsout}, 
and it was only recently shown \cite{litonly-difference} that in certain cases this restriction does not matter (we refer to the next 
section for a full discussion). We provide a counterexample that, interestingly, shows that our result is not generally valid if we eliminate the sufficient condition. 

\iffalse
For simplicity of modular arithmetic, in this paper we will assume that $p\geq 3$ is a prime (so that 
${\bf Z}_{p}$ is a field). This is most likely not an essential assumption, but it will make some of the tricks employed (e.g. 
Observation 2 below) easier. 
\fi 

Throughout the paper we will assume that $p\geq 3$ is an odd number. 
\section{Related Work} 

As mentioned in the introduction, the dynamics studied in \cite{istrate-balance} 
is a generalization to hypergraphs of a variant of the {\em lights out ($\sigma$)-game} \cite{sutner-lightsout}, a problem 
that has seen significant investigation. The version we considered in \cite{istrate-balance} is the 
apparently more 
constrained {\em lit-only $\sigma^{+}$-game}: 

\begin{definition} 
 Let $G=(V,E)$ be a finite graph. Each vertex $v\in V$ has a lightbulb (that is either ''on`` or ''off``) and a light switch. In the 
{\em lights out ($\sigma$)-game} pressing the light switch at any given vertex $v$ changes the state of the lightbulbs at {\em all 
neighbors of $v$}. In the $\sigma^{+}$-game the action also changes the state of the lightbulb at $v$. The {\em lit-only} versions of the $\sigma$ and $\sigma^{+}$ games only allow toggling switches of lit vertices. 
\end{definition}

Sutner \cite{sutner-lightsout} showed that the all zeros state is reachable from the all-ones state in the 
$\sigma^{+}$-game. This was generalized to Scherphuis \cite{scherphuis-lightsout} to the lit-only $\sigma^{+}$-game. A 
 recent result (\cite{litonly-difference} Theorem 3) significantly overlaps with our result in 
\cite{istrate-balance}, essentially showing that the lit-restriction does not make a difference for reachability on 
 hypergraphs that arise as so-called {\em neighborhood hypergraphs} \cite{neighborhood-hypergraph} of a given graph: \
\begin{definition} 
 Given graph $G=(V,E)$, the {\em neighborhood hypergraph of $G$}, denoted $N(G)$, is the hypergraph whose vertices are those 
of $G$ and whose edges correspond to sets 
\[
 N^{+}(v)=\{v\}\cup \{w\neq v\in V: v\sim w\}
\]
\end{definition}

The result in \cite{litonly-difference} is incomparable to our result in \cite{istrate-balance}, as it does not require, as we do, that the degree of each hyperedge to be at least three; 
on the other hand we do not restrict ourselves to neighborhood hypergraphs.

\iffalse
\begin{observation} 
A comparative lecture of \cite{istrate-balance} and \cite{litonly-difference} could seemingly arrive at the conclusion that our result 
in \cite{istrate-balance}, the one we aim to generalize in this paper, is inconsistent with that in \cite{litonly-difference}, in that it omits an important condition in the statement of Theorem 3 in \cite{litonly-difference}. 

This required condition is that the final state should not be ''all ones`` on any connected component. The difference in statements is, however, only apparent: a simple argument shows that 
if the final state $w_{2}$ is the ''all ones`` state then the only state $w_{1}$ for which the result in \cite{istrate-balance} allows proving reachability is $w_{1}=w_{2}$. 

However, for $p\geq 3$, this condition {\em does} make a difference, and has to be explicitly imposed to get our result. 
\end{observation} 
\fi 

A related operation on graphs called 
{\em Seidel switching} also yield dynamical systems  related to the one considered in this paper. Given a graph $G=(V,E)$, a {\em Seidel switching} at a node $v\in V$ yields a graph $H$ obtained by deleting from $G$ all edges $(v,w)\in E$ and adding to $H$ all edges $(v,w)\not \in H$. Mapping all edges to vertices of a 
Seidel switching  has recently been connected to lights-out games and investigated under random update \cite{random-seidel-switching}. Again, in contrast to such processes, the moves allowed in our systems correspond to "lit only" cases.

Lights out games were considered for finite fields ${\bf Z}_{p}$, $p\neq 2$ as well, e.g. in \cite{modular-domination}.
Our framework differs from the one in that paper in several important ways: first we consider the $\sigma^{+}$-game (rather than 
the $\sigma$-game). Second our definition differs slightly in the specification of the dynamics, as the value of the scheduled vertex 
{\em decreases}, rather than increases, by one (as it does in \cite{modular-domination}). The motivation for this variation is our desired connection 
with the theory of {\em interacting particle systems} \cite{liggett-ips}, particularly with the definition of coalescing/annihilating 
random walks. 

The further connections with this latter theory are also worth mentioning: threshold coalescing 
and annihilating random walks, where several particles have to be present at a site for interaction 
with the new particle to occur, have previously been studied (e.g. \cite{stephenson-threshold}) in the interacting particle 
systems literature. Compared to this work our results differ in an important respect: instead of working 
on a lattice like ${\bf Z}^{d}$ our result considers the case of a finite hypergraph. Remarkably few 
results in this area (e.g. \cite{donnelly-welsh-coloring}, \cite{donnelly-welsh-interacting}, see 
also \cite{aldous-fill-book} Chapter 14) consider the case of a finite graph topology, 
much less that of a finite hypergraph.

Finally, we briefly discuss the connections between the dynamical model studied in this paper and the Statistical Physics of social dynamics \cite{castellano2009statistical}.  
As stated, a model inspired by the sociological theory of {\em social balance} \cite{heider-book} that originated in the Statistical Physics literature \cite{redner-balance,antal2005dynamics}, was the original motivation for our work \cite{istrate-balance}. The model in this paper shares with the one in \cite{istrate-balance} a similar relationship to the one between the Potts and Ising models. In our case, however, we do not see how to sensibly extend the model in \cite{redner-balance,antal2005dynamics} so that it corresponds to our generalization of annihilating random 
walks. On the other hand such walks correspond via {\em duality} (see \cite{aldous-fill-book} Chapter 14 and \cite{griffeath-ips}) to a fundamental 
model of opinion dynamics, the {\em antivoter model}. ''Cyclic'' extensions of antivoter models have been investigated as well 
(e.g. \cite{griffeath1987survival,bramson1989flux}), and we can probably define such a ``cyclic'' extension that corresponds via duality to our ${\bf Z}_{p}$-generalization of 
annihilating random walks. This (and a more complete study of our system as a Markov chain) are left for subsequent work. 
                                                                                                           
\section{The Main Result}

Assume $|V(G)|=n$, $|E(G)|=m$,  and let $w_{1}$ and $w_{2}\in {\bf Z}_{p}^{n}$ be states of the system such that $w_{2}$ is reachable from $w_{1}$. 
Define variables $X_{e,v}$ denoting the number of times (modulo $p$) that vertex $v$ and hyperedge $e$ are chosen 
in the process from Definition~\ref{maindef}. The effect of scheduling pair $(v,e)$, given current configuration $w$, 
is to modify the value of $w[v[$ by $-1$ and of all $w[u]$, $u\in e\setminus \{v\}$, by +1 (mod p). Hence: 

\begin{equation}\label{diff} 
 \sum_{\stackrel{v\neq u,}{v,u\in e}}X_{e,u}- \sum_{v\in e}X_{e,v}=w_{2}[v]-w_{1}[v] \mbox{ }(mod\mbox{ } p) 
\end{equation} 

We will denote by $H(w_{1},w_{2},G)$ the system of equations~(\ref{diff}).

Does the converse hold? I.e. is the solvability of system $H(w_{1},w_{2},G)$ sufficient for the state $w_{2}$ to be reachable from 
$w_{1}$ ? The answer is easily seen to be negative: for any $p\geq 3$ state $(-1,-1,\ldots, -1)$ is a {\em Garden of Eden state} (that is, it has no preimage). 

%\begin{observation}\label{counter} 
What if add condition that $w_{2}\neq (-1,-1,\ldots, -1)$ ? The 
 answer to the previous question is still negative: Figure~1 provides a counterexample: from state $w_{1}=[0;\mbox{ } 1;\mbox{ } 1]$ one cannot reach state $w_{2}=[1;\mbox{ } 2;\mbox{ } 2]$, even 
though the system has a solution in ${\bf Z}_{3}$. Indeed, the only other configurations that can reach  $[1;\mbox{ } 2;\mbox{ } 2]$ are easily seen 
to be $[2;\mbox{ } 1;\mbox{ } 2]$ and $[2;\mbox{ } 2;\mbox{ } 1]$. 
%\end{observation} 

\begin{figure}
\begin{center} 
\begin{minipage}{5cm}
\includegraphics[width=4.5cm,height=4.5cm]{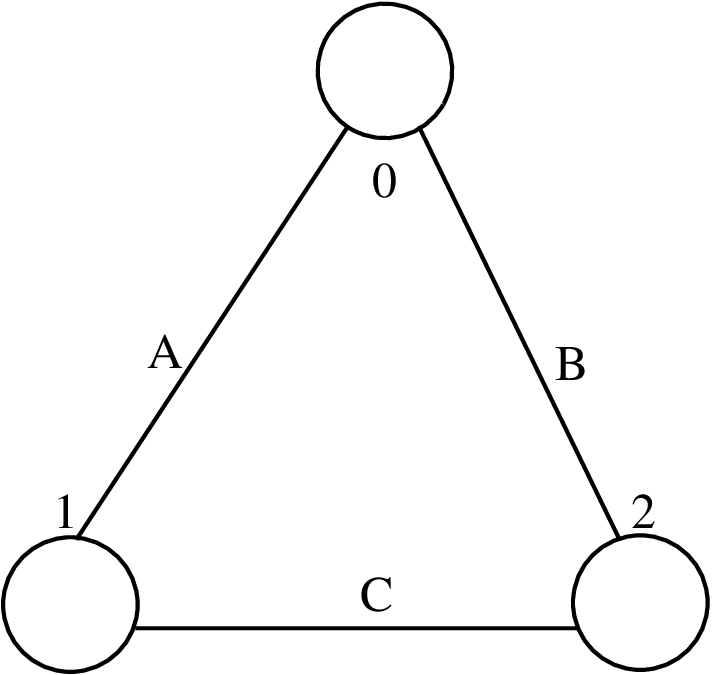} 
\end{minipage}
\begin{minipage}{5cm}
\[
[0;1;1]\rightarrow [1;2;2].\\
\]
\[
 \left\{\begin{array}{cc}              
              x_{1,A}-x_{0,A}+x_{2,B}-x_{0,B}=1, \\
              x_{0,A}-x_{1,A}+x_{2,C}-x_{1,C}=1, \\
            x_{1,C}-x_{2,C}+x_{0,B}-x_{2,B}=1, \\
              \end{array}
       \right.
\]
\[
\left\{\begin{array}{cc}
              x_{0,A}=0, 
           x_{0,B}=0, \\
 x_{1,A}=2, 
 x_{1,C}=0, \\
 x_{2,B}=2, 
 x_{2,C}=0. \\
              \end{array}
       \right.
\]
\end{minipage}
\end{center} 
\label{fig1}
\caption{(a). The counterexample $G$ (p=3) (b). The system and its solution}
\end{figure}

The counterexample in Figure~\ref{fig1} was a graph (had all hyperedges of cardinality 2). Restricting ourselves to hyperedges of size at least three  a converse does actually hold: 

\begin{definition}\label{goodh} 
 A connected hypergraph $G$ is {\em good} if for every hyperedge $e\in E(G)$, $|e|\geq 3$.
%\item %For every two hyperedges $e_{1}\neq e_{2}\in E(G)$, $|e_{1}\cap e_{2}|\leq 1$. In %particular, 
%$G$ is {\em simple}, i.e. for no two hyperedges $e_{1},e_{2}$ it 
%holds that $e_{1}\subseteq e_{2}$\footnote{In \cite{istrate-balance} this condition was implicitly assumed, being true for hypergraphs %arising via so-called {\em 
%triadic duality} from the original problem about social balance.}. 
%\end{itemize}
\end{definition}

Imposing the conditions in Definition~\ref{goodh} we obtain our main result. 

\begin{theorem}\label{trees}
 Let $G$ be a good hypergraph. Let $w_{1}$ be an initial configuration that is not identical to the "all zeros" configuration {\bf 0},
 and let $w_{2}$ be a final configuration, $w_{2}\neq [-1;-1;\ldots;-1]$. 

Then $w_{2}$ is reachable from state $w_{1}$ if and only if system $H(w_{1},w_{2},G)$ has a solution in ${\bf Z}_{p}$. 
\end{theorem}

\subsection{A Comment on the Significance of our main result}

The reader may wonder did we restrict ourselves in the statement of Theorem~\ref{trees} to good hypergraphs ? After 
all, we do {\bf not} expect that the notion of good hypergraphs captures all cases  for which a connection such as the one displayed in the theorem holds. 

The answer is that good hypergraphs form in some sense a natural {\em maximal class}: as shown by the example above, extending the result beyond good hypergraphs is impossible without further complications in the statement. 

\section{Proof of the Main Theorem} 

We will need the following definitions: 

\begin{definition}
 When system $H(w_{1},w_{2},G)$ is solvable we define {\em the norm of the system $H(w_{1},w_{2},G)$} as the quantity 
 \[
 |H(w_{1},w_{2},G)|=\min\{y_{1}+y_{2}+\ldots +y_{nm}\} 
 \] 
 where $y=(y_{1},\ldots, y_{nm})$ ranges over the (finite set) of all solutions in ${\bf Z}_{p}$ of the system, but when taking the sum above the $y_{i}$'s are interpreted as {\em integers} in $\{0,\ldots, p-1\}$, rather than in ${\bf Z}_{p}$. 
\end{definition} 

\begin{definition} 
Also, define the {\em width of the system $H(w_{1},w_{2},G)$} to be the minimum (over all solutions $x$ of the system) of 
\[
 |\{e\in E(G):\exists v\in e\mbox{ }|\mbox{ } x_{v,e}\neq 0\}|. 
\]
\end{definition} 

\begin{definition} 
Let $G=(V,E)$ be a hypergraph, $l\in E$ be a hyperedge in $G$, and $v\in l$ a vertex. We define state vector 
$a_{v,l}\in {\bf Z}_{p}^{n}$ by 
\[
a_{v,l}[z]=\left\{\begin{array}{ll}
              +1& \mbox{, if }z=v, \\
              -1& \mbox{, if }z\neq v, z\in l, \\
              0 & \mbox{, otherwise.}\\
              \end{array}
\right.
\]
\end{definition} 

\begin{definition}
Let $G=(V,E)$ be a hypergraph, $l\in E$ be a hyperedge in $G$, $w\in {\bf Z}_{p}^{n}$ be a state and $b\in {\bf Z}_{p}^{n}$. We denote by $w^{[b,l]}$ the following state: 
\[
w^{[b,l]}[v]=\left\{\begin{array}{ll}
              w[v]& \mbox{, if }v\not \in l, \\
              w[v]+b(v) & \mbox{, otherwise.}\\
              \end{array}
       \right.
\] 

Also, with the conventions in the previous definition, we will write $w^{[v,l]}$ instead of $w^{[a_{v,l},l]}$ and, for $k\geq 1$, 
$w^{[k,v,l]}$ instead of $w^{[k\cdot a_{v,l},l]}$. Vector $w^{[k,v,l]}$ can be interpreted as applying k moves at vertex v on edge l. 
\end{definition} 

\begin{definition}\label{good} 
 A pair of vertices $(v_{1},v_{2})$ is {\em good in state $w$} if $(w[v_{1}],w[v_{2}])\not\in \{(0,0),(-1,-1)\}$. 
\end{definition}

We first make the following simple 

\begin{observation}
\label{trick} 
Let $C$ be a configuration on hypergraph $G$ and $v_{1}\neq v_{2}$ 
two vertices of $G$ in the same hyperedge $e$ such that pair $(v_{1},v_{2})$ is good in $C$. Then one can
 change configuration $C$ into configuration $D$ that has the same number of particles at $v_{1},v_{2}$ but the number of particles at any other 
vertex $v$ of $e$ increases by one (mod $p$). The move only involves edge $e$ and some of its vertices.
A similar statement holds for decreasing labels by one (mod $p$), instead of increasing them.  
\end{observation} 

\begin{proof} 
If $label(v_{1})\neq 0$ and $label(v_{2})\neq p-1$ first make a move at vertex $v_{1}$ then make a move at vertex $v_{2}$. The number of particles at $v_{1},v_{2}$ stays the same, whereas it increases by two (mod p) at any other vertex. Since $p\geq 3$ is odd, $p$ is relatively prime to 2. We repeat this process $\lambda$ times, where $\lambda$ is chosen such that $2\lambda=1$ (mod p). 
If $label(v_{2})=-1$ then $label(v_{1})\neq -1$ (mod p), so we may repeat the above scheme with moves first made at $v_{2}$ then at $v_{1}$. 

The proof for the second case is identical, with $2\lambda = -1$ (mod $p$). 
\end{proof} 

\begin{proof} 
We prove Theorem~\ref{trees} by induction on $m$, 
%the number of hyperedges of $G$. 
the width of system $H(w_{1},w_{2},G)$. The proof, although simple, is somewhat cumbersome, comprising a large number of cases with several subcases of their own. For ease of comprehension, a visual outline of the proof and the various dependencies between the intermediate results is presented in Figure~2. 

\begin{figure}
\begin{center} 
\includegraphics[width=10cm,height=6cm]{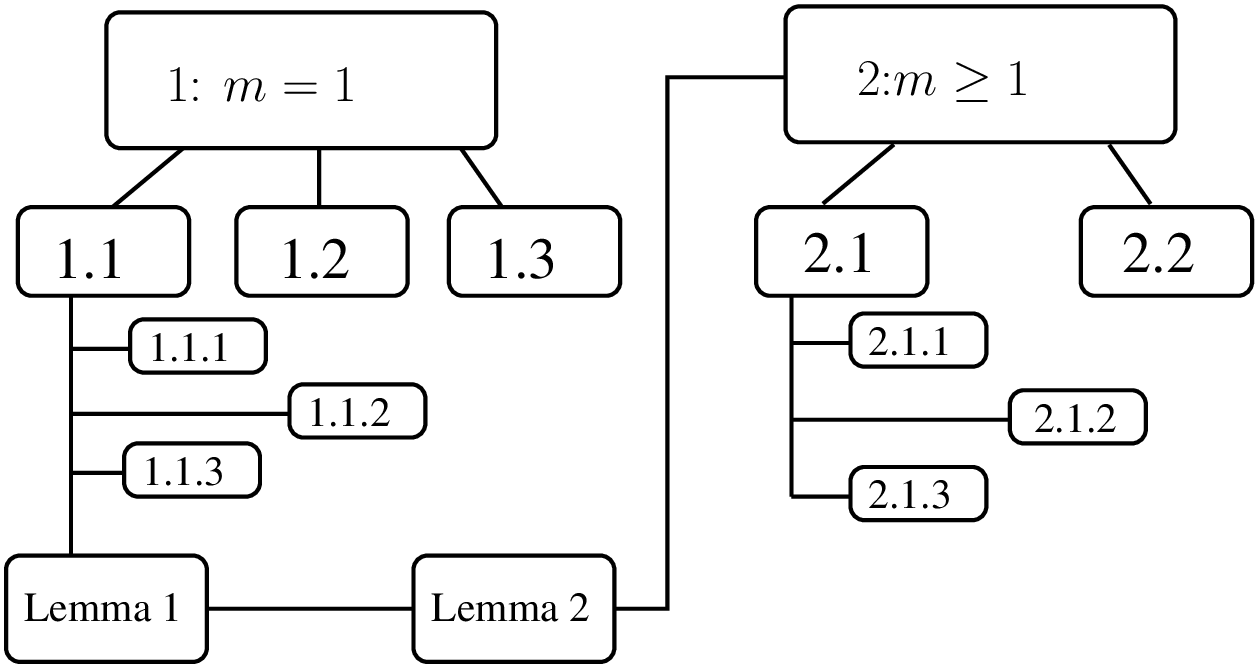} 
\end{center} 
\label{fig-outline}
\caption{Logical flow  of the proof.}
\end{figure}

\begin{itemize}
 \item{\bf Case 1: $m=1$. } Suppose system $H(w_{1},w_{2},G)$ has a solution of width one, thus involving a single edge $e$ of 
$G$. We infer that  
$w_{2}[v]=w_{1}[v]$ for all vertices $v\not \in e$ (otherwise the system would contain an equation $0=\lambda$, with 
$0\neq \lambda = w_{2}[v]-w_{1}[v]\in {\bf Z}_{p}\setminus \{0\}$). 

We will prove this case using two subcases, depending whether the restriction of state 
$w_{1}$ to hyperedge $e$, denoted by $w_{1}|_{e}$, is identically zero or not. 

\begin{itemize} 
 \item{\bf Subcase 1.1: $w_{1}|_{e}\neq {\bf 0}$ and $w_{2}|_{e}\neq [-1;-1;\ldots -1]$.}

We will give a solution involving only vertices of edge $e$. Since $w_{1}=w_{2}$ outside $e$ we can 
 assume that $G$ consists of exactly those vertices 
$v_{1},v_{2},\ldots, v_{k}$ connected by edge $e$. Denote $\overline{w}$ the vector $w_{2}-w_{1}$ and, for simplicity, 
let $\overline{w_{1}},\overline{w_{2}},\ldots, \overline{w_{k}}$ be shorthands for $\overline{w}[v_{1}],\overline{w}[v_{2}],\ldots, \overline{w}[v_{k}]$. 
Similarly, let $w_{a,i}$ stand for $w_{a}[v_{i}]$, where $i=1,\ldots,k$, $a=1,2$. 
Also define $w=\overline{w_{1}}+\overline{w_{2}}+\ldots + \overline{w_{k}}$. System $H(w_{1},w_{2},G)$ reads: 
\[
\left\{\begin{array}{cc}
           -x_{1}+x_{2}+\ldots + x_{k}= \overline{w_{1}}\\ 
x_{1}-x_{2}+\ldots + x_{k}= \overline{w_{2}}\\ 
\ldots \\
x_{1}+x_{2}+\ldots - x_{k}= \overline{w_{k}}\\ 
              \end{array}
       \right.
\]

The solvability (and solutions) of system $H(w_{1},w_{2},G)$ can easily be characterized in this case, and depends on whether $p$ divides or not $(k-2)$. In the latter case one can easily check that for any $w_{1},w_{2}$ system $H(w_{1},w_{2},G)$ has an unique solution 
$x_{i}=2^{-1}[(k-2)^{-1}w - \overline{w_{i}}]$, $i=1,\ldots, k$. In the former case, the system $H(w_{1},w_{2},G)$ has a solution if and only if $w=0$ (mod p). Indeed, the condition 
follows immediately from summing the equations of the system. On the other hand if $w=0$ holds one can easily verify that 
the following family 
\[
\left\{\begin{array}{ll}
           x_{1}= \lambda\\ 
           x_{2}= \lambda+2^{-1}(\overline{w_{1}}-\overline{w_{2}})\\ 
           x_{3}= \lambda+2^{-1}(\overline{w_{1}}-\overline{w_{3}})\\ 
\ldots \\
           x_{k}= \lambda+2^{-1}(\overline{w_{1}}-\overline{w_{k}})\\ 
              \end{array}
       \right.
\]
with $\lambda$ arbitrary in ${\bf Z}_{p}$, represents the family of solutions of system $H(w_{1},w_{2},G)$. 

In what follows we will not refer to this case dichotomy, but will simply prove the result by induction over $v=|H(w_{1},w_{2},G)|$. W%ith little risk of ambiguity, w
e will also denote $x=(x_{1},\ldots, x_{nm})$ a solution of $H(w_{1},w_{2},G)$ that witnesses the fact that the minimum in the definition of $|H(w_{1},w_{2},G)|$ is equal to $v$. 

\begin{itemize} 
\item {\bf Case 1.1.1: $v=1$. } 

Let $i_{0}$ be the unique index such that $x_{i}\neq 0$. Then $w_{2,i_{0}}=w_{1,i_{0}}-1$ and $w_{2,j}=w_{1,j}+1$ 
for $j\neq i_{0}$, the equalities being interpreted in ${\bf Z}_{p}$. In other words, we need to show how to change state vector 
$[w_{1,1};w_{1,2};\ldots ;w_{1,k}]$ into state vector 
$[(w_{1,1}+1);(w_{1,2}+1);\ldots (w_{1,i_{0}}-1);\ldots ;(w_{1,k}+1)]$.

If $w_{1,i_{0}}\neq 0$ a simple move at $v_{i_{0}}$ changes state $w_{1}$ into $w_{2}$ directly. So the only   
case that needs a proof is $w_{1,i_{0}}=0$. 

Let $j\neq i_{0}$ such that $w_{1,j}\neq 0$. Such an index exists since $w_{1}\neq {\bf 0}$. Furthermore, by reassigning indices we 
may assume without loss of generality that $i_{0}=1$ and $j=2$. Thus target state vector is $[(p-1);(w_{1,2}+1);\ldots ;(w_{1,k}+1)]$
\begin{enumerate} 
\item First, using $r$ times the trick from Observation~\ref{trick} at vertices $v_{1}$ and $v_{2}$  changes state $w_{1}=[0; w_{1,2};\ldots ;w_{1,k}]$ into state $[0; w_{1,2};w_{1,3}+r); \ldots ;(w_{1,k}+r)]$. We choose 
$r\in \{0,1\}$ (mod p) in such a way so that $w_{1,3}+r\neq 0,(p-1)$ (mod p). Next, apply (p-2) times the trick in Observation~\ref{trick}  between vertices $v_{2}$ and $v_{3}$ to turn the state vector into $[-2; w_{1,2};(w_{1,3}+r); \ldots ;(w_{1,k}+r-2)]$. 

Apply now a move at $v_{3}$ to turn the state vector into $[(p-1);(w_{1,2}+1);(w_{1,3}+r-1); \ldots ;(w_{1,k}+r-1)]$. 

\item If $w_{1,2}+1\neq (p-1)$ (mod p) then by applying $2-r$ times (mod p) the trick in Observation~\ref{trick}  to vertices $v_{1}$ and $v_{2}$ we reach the desired final state.  
\item Suppose we cannot reach alternative 2 for {\bf any} choice of $j$ with $w_{1,j}\neq 0$. Therefore, vector $w_{1}$ contains only zeros and $(p-2)'s$, with at least one $(p-2)$. Rearranging indices, we may assume that $w_{1}=[0;(p-2);0^{r-1};(p-2)^{k-r-1}]$, for some $1\leq r\leq k-1$, and that the target vector 
is $w_{2}=[(p-1);(p-1);1^{r-1};(p-1)^{k-r-1}]$. 

It is easy to see that alternative $r=1$ is impossible, given the hypotheses of Subcase 1.1: assuming otherwise, we would have $w_{2}|_{e}=[-1;-1;\ldots;-1]$, contradicting the 
second hypothesis.  

On the other hand for $r\geq 2$ moving from $w_{1}$ to $w_{2}$ is easy: first use Observation~\ref{trick} between the first two vertices to turn $w_{1}$ into vector $w_{3}=[0;(p-2);2;2^{r-2};0^{k-r-1}]$. Then hold vertices $v_{2}$ and $v_{3}$ and use Observation~\ref{trick} again to turn vector $w_{3}$ into $w_{4}=[-2;-2;2;0^{r-2};
(-2)^{k-r-1}]$. Finally, a single move at the third vertex yields 
final state $w_{2}$. 

Note that the assumptions $w_{2}\neq (-1)^{k}$ and $k\geq 3$ are the properties that allowed us to 
conclude that $r\geq 2$, ultimately enabling the construction above.
This is the step of the proof that critically employs these assumptions. 
\end{enumerate}

\item {\bf Case 1.1.2: $v\geq 2$. } If one of the following two conditions hold 
\begin{itemize} 
\item there exist two indices $i$ with $w_{1,i}\neq 0$, or 
\item only one such index exists, but a single move at $v_{i}$ moves the configuration to $w_{3}\neq {\bf 0}$
\end{itemize} 

then we first make one available move that brings the system to $w_{3}$. Now it is easily checked that system $H(w_{3},w_{2},G)$ is solvable and has norm $v$-1;  we apply the induction hypothesis. 

The only remaining case is $w_{1}=[1;-1;\ldots ;-1]$ and $w_{2}=[(1-v);(v-1);\ldots ;(v-1)]$. This is easily solved: First 
apply $2v$ times the trick in Observation~\ref{trick} to vertices $v_{1}$ and $v_{2}$ in order to change the state of the system to 
$[1;-1;(2v-1);\ldots ;(2v-1)]$. Then make a move $p-v\mbox{(mod p)}$ times at $v_{2}$. 

This concludes the proof of the case 1.1.2 and, with it, of Subcase 1.1.
\qed

\end{itemize} 

Before continuing with the remaining subcases of Case $1$, we give two applications 
of Subcase 1.1, namely  
Lemmas~\ref{first} and \ref{second} below, that will be useful in the sequel: 

\begin{lemma}\label{first} 
 Assume that $w_{1},w_{2}$ are states differing only on hyperedge $e$ 
whose restrictions to this edge are different from both $(0;0;\ldots;0)$ and $(-1;-1;\ldots;-1)$.  

Further assume that $w_{2}$ is reachable from $w_{1}$ via moves of 
edge $e$ only. Then $w_{1}$ is reachable in this way from $w_{2}$ as well. That is, we can ``undo'' a 
 sequence of moves on a given edge as long as the initial and the final states are both nonzero and different from $(-1;-1;\ldots;-1)$.  
\end{lemma}
\begin{proof} 
We can simply reason in the hypergraph $G_{2}$ containing edge $e$ only. Since $w_{2}$ is reachable from $w_{1}$, system $H(w_{1},w_{2},G_{2})$ has 
a solution $u$. It is easy to see that $-u$ is a solution to $H(w_{2},w_{1},G_{2})$ and we apply the result proved in Subcase 1.1.   
\end{proof} 

We next generalize the preceding lemma to the case when the hypergraph does not consist of a single edge anymore. To do so we need the following: 

\begin{definition} 
Given a hypergraph $H=(V_{H},E_{H})$, a {\em simple path} is a sequence of edges $Q=(q_{1},q_{2},\ldots q_{m})$ such that for all $1\leq i\neq j\leq m$ $q_{i}\cap q_{j}=\emptyset$ unless $j=i\pm 1$, in which case $q_{i}\cap q_{j}\neq \emptyset$.  
\end{definition} 

The desired generalization is: 

\begin{lemma}\label{second} 
 Let $P=(e_{1},e_{2},\ldots, e_{k})$, $k\geq 3$ be a simple path in hypergraph $G$. Let $s_{1}$ be a state such that there exists $s_{1}|_{e_{1}\setminus e_{2}}|\not \equiv {\bf 0}$. For $i=1,\ldots, k-1$ let $V_{i+1}=e_{i}\cap e_{i+1}$, let $v_{i+1}\in V_{i+1}$, %let $b=w_{1}(v_{k})$, 
and assume that $s_{1}[z]=0$ for all 
$z\in e_{2},e_{3},\ldots, e_{k-1}$.%\setminus V_{k}$. 
Then there exists vertex $v_{1}\in e_{1}\setminus e_{2}$ with $s_{1}(v_{1})\neq 0$ such that configuration 
$s_{2}$, specified by 
\[
 s_{2}[v]=\left\{\begin{array}{cc}
               s_{1}[v]-1,&\mbox { if }v= v_{1}, \\
               s_{1}[v]+1,&\mbox { if }v\in e_{1},v\neq v_{1},v\not \in e_{2}\\
               0,&\mbox { if }v\in \{v_{2},\ldots, v_{k-1}\}\\%, v\neq v_{k}, \\
               2 ,&\mbox{ if v is a vertex in }V_{i}\setminus \{v_{i}\}, 2\leq i\leq k-1,\\ 
               1 ,&\mbox{ if v is another vertex in one of }e_{2},\ldots, e_{k-1}\\
                1,&\mbox { if }v\in V_{k}, \\
                s_{1}[v] ,&\mbox{ otherwise. }\\
              \end{array}
       \right.
\]
 is reachable from $s_{1}$ (and viceversa) by making moves only along path $P$. 
 \end{lemma} 
 
 We generalize the forward process in Lemma~\ref{second} to a forward-backward process as follows: 
 
 \begin{lemma}\label{third} 
 It is possible to perform a set of forward moves on path $P$, similar to that described in Lemma~\ref{second}, such that if we subsequently we perform the following transformation: 
\begin{enumerate} 
\item We change the values of nodes $v\in e_{k}\setminus e_{k-1}$ to $\mu[v]$, and of those in $e_{k}\cap e_{k-1}$ to $1+\mu[v]$,  where $\mu[v]\in {\bf Z}_{p}$ (denote by $s_{3}$ the resulting state). We assume that these changes are performed without modifying the values of any node in $P\setminus e_{k}$.  
\end{enumerate} 
 then we can perform restoring moves on edges in $P\setminus e_{k}$ to bring back all values of this path $P$ to their values in $w_{1}$, except for nodes $v\in V_{k}$ for which the final value will be %$w_{1}(v)+]
$\mu[v]$. 
\end{lemma}
 
Lemma~\ref{third} informally states that one can ``propagate a one`` along the path from $v_{1}$ to $v_{k}$ as long as vertices between the two are initially zero, 
and then restore the configuration 
(see Figures~\ref{forward} and~\ref{backward}, in which $\lambda =0$ and all sets $V_{i}$ have cardinality 1). 

\begin{proof} 
The forward moves are easy: just choose $v_{1}$ arbitrarily in $e_{1}\setminus e_{2}$ with $s_{1}[v_{1}]\neq 0$. Then schedule, in turn, vertices $v_{1},v_{2},\ldots,$ $\ldots, v_{k-1}$, on edges $e_{1},\ldots, e_{k-1}$ respectively in this order. We use the 
fact that labels of $v_{2}, \ldots, v_{k-1}$ are initially zero, hence scheduling them in turn increases the label of the next node ($v_{k}$, in case of 
the last one) by one. The new nodes (except maybe the last) get values equal to one, so they can be scheduled in turn. Vertices that are ''internal`` 
to one of the edges $e_{2},\ldots, e_{k-1}$ get value 1; vertices in $V_{i}\setminus v_{i}$ (if any) get value 2. 

Suppose now that values of vertices in $V_{k}$ have been altered in the way specified by Step 2 of the multi-level process, resulting in state $s_{3}$. 

The analysis of the backward schedule is only a little more complicated, and comprises three cases: 
\begin{itemize} 
\item{\bf Case I: $s_{3}|_{e_{k-1}}\neq {\bf 0}$.}  

In this case we can also choose $v_{1}$ arbitrarily with the above constraints. 
First  we ''undo`` in succession the forward moves on sets $e_{k-1}\setminus V_{k},e_{k-2},\ldots, e_{2}\setminus V_{2}$, turning nodes on in these sets to zero, and nodes in $V_{2}$ to 1. To do so we use Case 1 of the Theorem and the fact that each $e_{i}$, $2\leq i\leq k-2$ contains at least one ''internal`` node (whose label is 1), or a node in $V_{i}$, whose label is nonzero. The proof of this last claim crucially uses the  conditions in the definition of good hypergraphs. 
For edge $e_{k-1}$ the argument uses the fact that $s_{3}(v_{k-1})=0$ and $s_{3}|_{V_{k}}\neq 0$. 

We are left with vertices of $V_{2}$ with a label of 1. We can use it to restore the correct values on edge $e_{1}$ as well.
\item{\bf Case II: $k\geq 4$, $s_{3}|_{e_{k-1}}\equiv {\bf 0}$.}  The argument is almost similar. First,  $e_{k-1}$  is already in the state we want to obtain, since $s_{1}|_{e_{k-1}}=s_{3}|_{e_{k-1}}={\bf 0}$.

Furthemore, as $s_{3}|_{e_{k-1}}\equiv {\bf 0}$ we infer the following things: 
\begin{enumerate} 
\item $V(e_{k-1})\subseteq V_{k-1}\cup V_{k}$ (otherwise any node in $e_{k-1}\setminus (V_{k-1}\cup V_{k})$ would have value 1 in $s_{3}$).
\item $|V_{k-1}|=1$ (otherwise any vertex in $V_{k-1}\setminus \{v_{k-1}\}$ would have label 2 in $s_{3}$).
\end{enumerate} 

Because of this second condition and $k\geq 4$, edge $e_{k-2}$ must have a node with nonzero value in $s_{3}$: as $|e_{k-2}|\geq 3$, either  $V_{k-2}$ has cardinality greater than one (and thus contains a vertex whose label is 2) or there exists a node 
``internal to $e_{k-2}$'', that is in $e_{k-2}\setminus (V_{k-2}\cup V_{k-1})$, whose label i $s_{3}$ is 1.  

In this case we can start the changing back process from $e_{k-2}$. 

\item{\bf Case III: $k=3$, $s_{3}|_{e_{2}}\equiv {\bf 0}$.}

The strategy will be to carefully choose $v_{1}$ in $e_{1}\setminus e_{2}$ and perform a {\em modified forward schedule} that will schedule edge $e_{1}$ twice, nd make vertex $v_{2}$ subsequently assume value $2$ instead of 1. Then when propagating on $e_{2}$ the label of $v_{2}$ will become 1 instead of zero. Hence propagation on edge $e_{2}$ can be undone in the backward phase.   

As in Case II by reasoning about the regular forward process we infer the following things: 
\begin{enumerate} 
\item $V(e_{2})\subseteq V_{2}\cup V_{3}$ (otherwise any node in $e_{2}\setminus (V_{2}\cup V_{3})$ would have value 1 in $s_{3}$).
\item $|V_{2}|=1$, i.e. $V_{2}=\{v_{2}\}$, otherwise any vertex in $V_{2}\setminus \{v_{2}\}$ would have label 2 in $s_{3}$. 
\end{enumerate} 

However, in this case, as $|V(e_{1})|\geq 3$ we also infer the following extra fact: 
\begin{equation} 
|e_{1}\setminus e_{2}|> 1.
\label{fact} 
\end{equation} 

The modified forward process is specified as follows: 
\begin{itemize}
\item[-]
If there exists a vertex $v_{1}\in e_{1}\setminus e_{2}$ with $s_{1}[v_{1}]\neq 0,1$ we choose such a vertex and make two moves from $v_{1}$ (instead of one) on the forward schedule. This will ensure the desired label for $v_{2}$ in the modified forward process. 
\item[-]
If there exists no such $v_{1}$ it means that $s_{1}$ only assumes values $0,1$ on $e_{1}\setminus e_{2}$. Then the modified process proceeds by first choosing $v_{1}$ with $s_{1}[v_{1}]=1$ and make a move on $e_{1}$ at $v_{1} $. Because of condition~(\ref{fact}) there now exists a vertex $v_{1}^{\prime}\in e_{1}\setminus e_{2}$ whose label is $1$ or $2$ (hence nonzero). We make the move on $e_{1}$ at $v_{1}^{\prime}$, bringing the label of $v_{2}$ to 2 as needed, and then continue with the forward process on $e_{2}$. 
\end{itemize}

%\textcolor{red}{De completat}
\end{itemize} 
\end{proof} 

\begin{observation} 
Lemma~\ref{third} assumed that path $P$ has length at least three. In fact we can extend the Lemma (in a slightly modified form) so that it applies to paths $P$ has length two if we are allowed to carefully choose vertex $v_{1}$. 

Specifically we need to chose it so that scheduling it will produce at least one nonzero vertex in $e_{1}\setminus e_{2}$. The only problematic case is when there exists an unique vertex in $e_{1}\setminus e_{2}$ having label 1 in $s_{1}$ and all other vertices have label 0. In this case $|V_{2}|\geq 2$. Let $u_{2}\neq v_{2}$ be such a vertex. Performing the trick of Observation~\ref{trick} on vertices $v_{1}$ and $u_{2}$ is enough to turn the label of all vertices of $V_{2}$, other than $u_{2}$ to 1, while preserving the label of $v_{1}$. This makes backward restoration unnecessary, as the label of $v_{1}$ was not affected. 
\end{observation}

We now return to the proof of the Case $1$ of Theorem~\ref{trees}, specifically to the remaining subcase:  

\item{\bf Subcase 1.2: $w_{1}|_{e}\equiv {\bf 0}$ but $w_{2}|_{e}\neq [-1;-1;\ldots -1]$.}

In this case we apply Lemma~\ref{third} to reduce the problem to Subcase 1.1 as follows: let $v$ be a vertex with $w_{1}(v)\neq 0$ 
at minimal distance from $e$. Let $P$ be a path of minimal length connecting $v$ to a set of vertices $U$ of $e$. Path $P$ is simple by minimality. 

We use the forward trick in Lemma~\ref{second}
to propagate a 1 value to vertices of $U$, thus making the resulting state nonzero on edge $e$. Then we use the case $w_{1}|_{e}\not \equiv {\bf 0}$ and the solvability of the resulting associated system 
to change the state of the system to $w_{3}$, where $w_{3}$ has the value prescribed above on $P\setminus e$ and $w_{3}=w_{2}$ on $e$ except at vertices $v\in U$, for which $w_{3}[v]=w_{2}[v]+1$.  

If $w_{3}|_{U}\neq {\bf 0}$ or $w_{3}|_{U}\equiv {\bf 0}$ but the last edge $f$ of $P$, the one that intersects $e$ on $U$, contains a vertex with label 1 or 2 then $w_{3}|_{f}\neq {\bf 0}$. Thus we may perform the multi-stage trick of Lemma~\ref{third}, where the vertices $v\in U$ have been modified by the amount $(w_{2}[v]+1)-1=w_{2}[v]$ to restore the state $w_{1}$ along $P$ except on $U$, where it will be $w_{1}+w_{2}|_{U}$, that is $w_{2}|_{U}$, since $w_{1}|_{U}\equiv {\bf 0}$. 
   
\item{\bf Subcase 1.3: $w_{1}|_{e}\equiv [0;0;\ldots 0]$ and $w_{2}|_{e}\equiv [-1;-1;\ldots -1]$.}

Let $v$ be a vertex (necessarily not in $e$) with $w_{1}(v)=w_{2}(v)\neq 0$ and let $P$ be a minimal path connecting $v$ to $e$. We use forward propagation to change the state of vertices in $e\cap P\neq \emptyset$ to 1. Then we use Subcase 1.1 to change the labels to state $w_{3}$, whose restriction to edge $e$ is as follows: $-1$ on $e\setminus P$, and $0$ on $e\cap P$. Finally, we change the state to $-1$ on $e\setminus P$ and restore the state on path $P\setminus e$ by using backward propagation. 

\end{itemize}

\begin{figure}[ht]
\hfill
\begin{center}
\includegraphics[width=10cm]{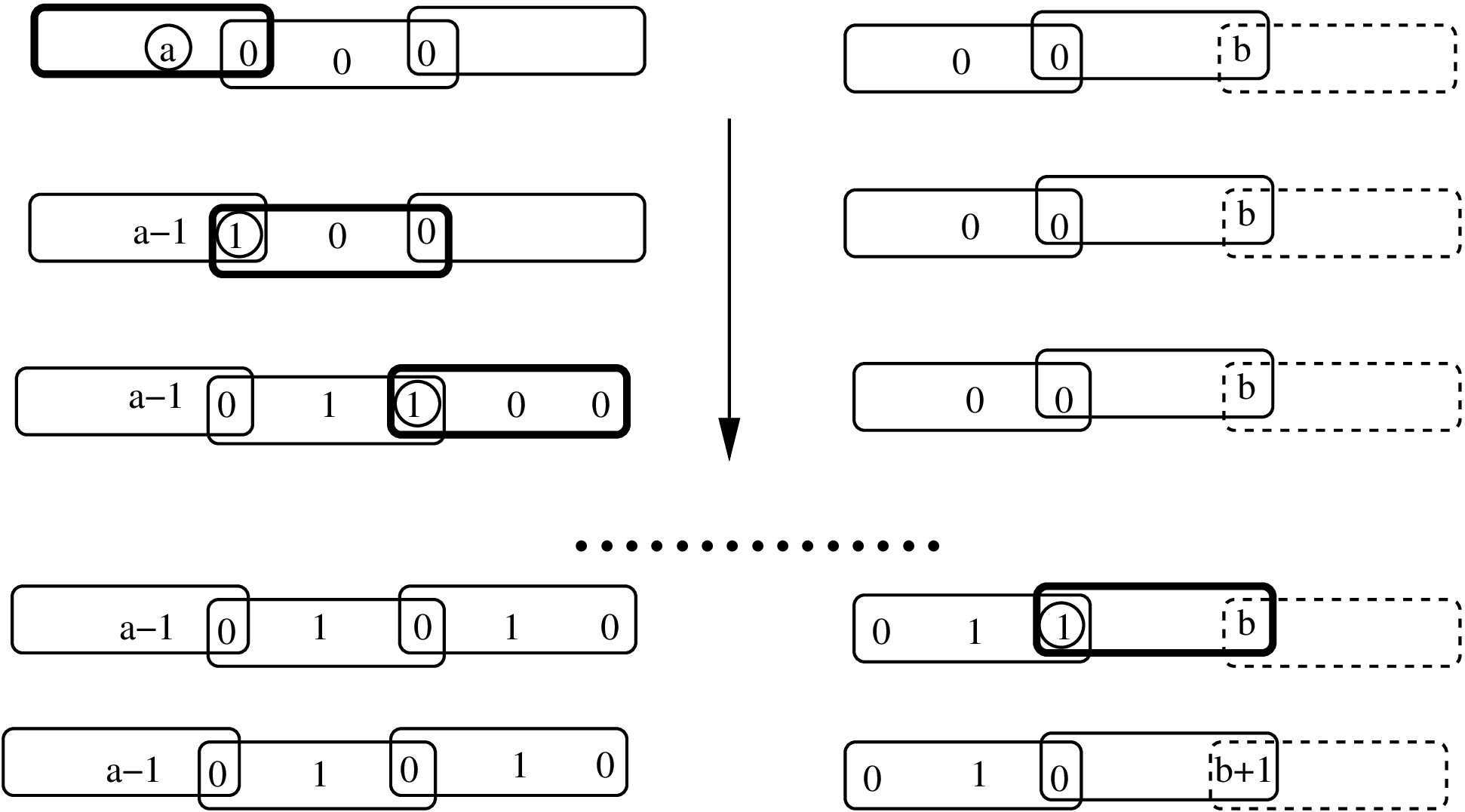}
\end{center}
\caption{Forward propagation of nonzero values. Initially node $v_{1}$ has value $a$ and 
target node $v_{k}$ has value $b$ (with edge $e_{k}$ being dashed). On each row the node scheduled at that stage is circled, with the scheduled edge being darkened. Scheduling a node has the effect of decreasing its label by one (mod p) and increasing the label of all other nodes in the edge by one (mod p). For simplicity we pictured the situation when no two hyperedges intersect at a set of cardinality larger than one. Thus no label 2 is created and all hyperedges have ``internal'' nodes (whose new state is 1). Last row represents the resulting state. 
}
\label{forward}
\end{figure}

\begin{figure}[ht]
\hfill
\label{backward}
\begin{center}
\includegraphics[width=10cm]{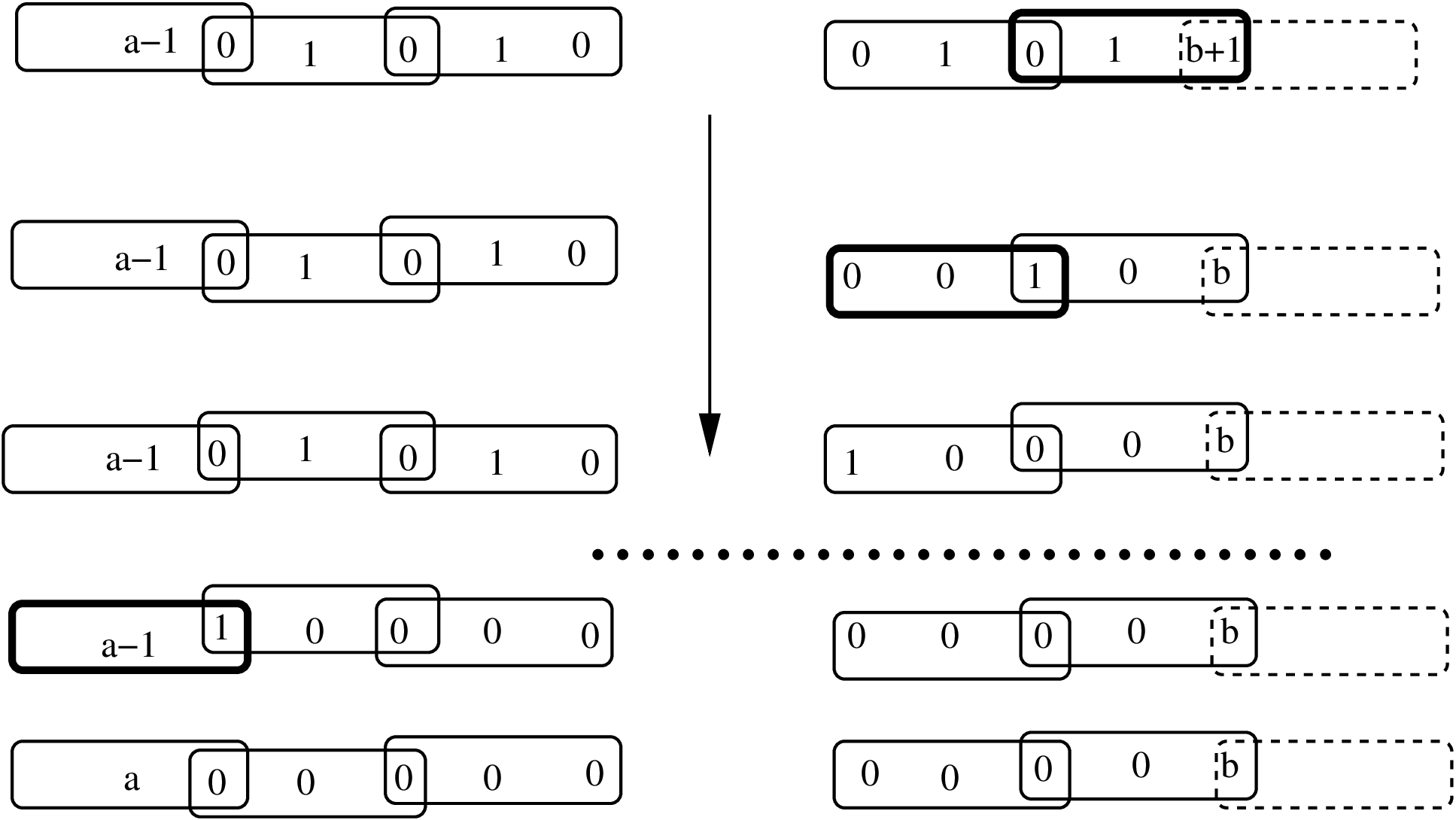}
\end{center}
\caption{Backward restoration. Each pictured ''step`` applies Observation~\ref{first} on one edge, starting from $e_{k-1}$ down to $e_{1}$. On each row the the scheduled edge is darkened.  For simplicity we pictured the particular situation when no two hyperedges intersect at a set of cardinality larger than one, hence all hyperedges have ``internal'' nodes (whose state is 1 after the forward moves). The scheduled nodes are these internal nodes. Last row represents the resulting state.}
\end{figure}

\item {\bf Case 2: $m\geq 2$ } 

Let $x$ be a solution of $H(w_{1},w_{2},G)$ of minimal width $m$, and $e_{1},e_{2},\ldots, e_{m}$ the edges of $x$ for 
which there exists a vertex $v\in e$ with $x_{v,e}\neq 0$. 

Let $w_{3,i}$ be the state of the system specified in the following way: $w_{3,i}(v)=w_{1}(v)+U_{i,v}$, with $U_{i}=(U_{i,v})_{v\in V}$ specified as follows: 

\begin{itemize}
\item $U_{i,v}=0$ for $v\not \in e_{i}$. 
\item For vertices $v$ of $e_{i}$ , $U_{i,v}=
-x_{v,e_{i}}+\sum_{w\neq v,w\in e_{i}} x_{w,e_{i}}$. 
\end{itemize} 

Intuitively $w_{3,i}$ is the state that would be reached from $w_{1}$ when ''making the moves specified by solution $x$ on edge $e_{i}$ only`` (if possible). 

There are two alternatives: 

\begin{enumerate} 
 \item{\bf Case 2.1: } There exists $1\leq i \leq m$ such that 
$w_{3,i}\neq {\bf 0}$ and $w_{3,i}\neq (-1;-1;\ldots;-1)$. 
\\
In this case system $H(w_{1},w_{3,i},G)$ has a solution $\overline{x}$ of width one, in fact zero outside edge $e_{i}$: it 
is simply $x$ with values outside of $e_{i}$ replaced by zeros. 

Applying Case $m=1$ of the theorem we infer that $w_{3,i}$ is reachable from $w_{1}$. Then it is easy to see that 
System $H(w_{3,i},w_{2},G)$ has width at most $m-1$, and we apply the induction hypothesis. 

\item{\bf Case 2.2: } $w_{3,i} \equiv {\bf 0}$ or $w_{3,i} \equiv (-1;-1;\ldots;-1)$ for all $1\leq i\leq m$. %We show that this alternative is impossible for $m\geq 2$. 
%\vspace{5mm}
There are three alternatives:
\begin{enumerate}
\item{\bf Case 2.2.1: } $w_{3,j}\equiv 0$ for all $j$. 
\item{\bf Case 2.2.2: }  $w_{3,j}\equiv (-1;-1;\ldots;-1)$ for all $j$. 
\item{\bf Case 2.2.3: } there exist $k\neq l$ such that $w_{3,k}\equiv {\bf 0}$ and $w_{3,l}\equiv (-1;-1;\ldots;-1)$.  
\end{enumerate} 

In case 2.2.3, the two relations above imply the fact that edges $e_k$ and $e_l$ cover all vertices in $G$ (otherwise there would be 
a vertex $v$ outside $e_{k}\cup e_{l}$, which is constrained to mutually incompatible values by the two relations).

 Since $G$ is connected, we may chose edges $e_{k},e_{l}$ to intersect (otherwise we would have two connected components induced by vertices in edges that intersect $e_{k}$, $e_{l}$, respectively). 

%Assume for now that $e_{k}\setminuseq e_{l}$. Conditions $w_{3,k}\equiv {\bf 0}$ and $w_{3,l}\equiv (p-1;p-1;\ldots;p-1)$ imply the following facts: 
%\begin{itemize}
%\item All vertices in $e_{l}\setminus e_{k}$ have the common value $0$ in $w_{1}$.  
%\end{itemize} 

Finally, state 
$w_{1}$ is determined, except on $e_k\cap e_l$: $w_{1}(v)=0$ if $v\in e_l\setminus e_k$, $w_{1}(v)=-1$ if $v\in e_k\setminus e_l$. 

In cases 2.2.1 and 2.2.2 the value of $w_{1}$ is determined on all vertices $v$ with the possible exception of vertices (if any) that are parts of all edges: $w_{1}(v)=0$ (Case 2.2.1), respectively $w_{1}(v)=-1$ (Case 2.2.2). 
\vspace{5mm} 

{\bf Case 2.2.1.} This conclusion implies the fact that 
$w_{1}\equiv {\bf 0}$, a contradiction, except in one case: that when all edges $e_{1},e_{2},\ldots, e_{m}$ intersect at vertices in some set $S$. 

We have to show that $m\geq 2$ is not possible even in this remaining case. Indeed, we further infer the fact that 
$w_{1}=w_{2}\equiv 0$ everywhere except $S$. On the other hand 
for every $v\in S$, any edge from $e_{1},\ldots, e_{m}$ turns the state of $v$ from $w_{1}(v)$ to zero. Their combined effect (represented by the system $H(w_{1},w_{2},G)$) is therefore such that 
then $w_{2}[v]= w_{1}[v] (1-m)$. 

In this case we show that system $H(w_{1},w_{2},G)$ has width one. We will be using a single edge, say $e_{1}$. Indeed, since $w_{3,1}\equiv 0$ 
it follows that  $y_{v,e_{1}}=-(p-m(mod\mbox{ }p))x_{v,e_{1}}$ for all 
$v\in e_{1}$, $y_{w,e}=0$, otherwise, is a solution of the system $H(w_{1},w_{2},G)$. 

\vspace{5mm} 

{\bf Case 2.2.2.} Similarly to case 2.2.1, this conclusion implies the fact that 
$w_{1}\equiv (-1;-1;\ldots;-1)$,  except in the case when all edges $e_{1},e_{2},\ldots, e_{m}$ intersect at vertices in some set $S$, in which case $w_{1}[v]=w_{2}[v]=-1$ for every $v\not \in S$.  

In this case we show that the case $m\geq 2$ is not possible either. The argument is similar to that of Case 2.2.1. The effect of every edge $e_{1},\ldots, e_{m}$ is determined by the condition $w_{3,i}\equiv (-1;-1;\ldots;-1)$: it leaves unchanged values on vertices outside $S$; for nodes $v\in S$ it changes value $w_{1}(v)$ to $-1$. The combined effect of all such edges (on vertices in $S$) is, therefore, to leave $w_{1}[v]$ unchanged outside $S$. On the other hand, on vertices $v\in S$ it change the value $w_{1}(v)$ to $w_{2}[v]=w_{1}[v]-m(w_{1}[v]+1).$ 

In this case again system $H(w_{1},w_{2},G)$ has width one: $y_{v,e_{1}}=-(p-m(mod\mbox{ }p))x_{v,e_{1}}$ for all 
$v\in e_{1}$, $y_{w,e}=0$, otherwise, is a solution. 

\vspace{5mm} 

{\bf Case 2.2.3.} Let $a,b\geq 1$ be the number of edges $e_{i}$ such that $w_{3,i}\equiv {\bf 0}$, $w_{3,i}\equiv (-1;-1;\ldots;-1)$ respectively. 

We claim that the system $H(w_{1},w_{2},G)$ has a solution of width at most two, equal to 
$a\cdot U_{k}+ b\cdot U_{l}$. This follows easily from using an idea similar to that of Cases 2.2.1 and 2.2.2: $a$ of the edges have a similar effect as making the move according to vector $U_{k}$. Their combined effect is therefore identical to that of $a\cdot U_{k}$. We reason similarly for the b edges whose effect is equal to changing state by vector $U_{l}$. The combined effect of all edges $e_{1},\ldots, e_{m}$ is thus equal to $a\cdot U_{k}+ b\cdot U_{l}$, which means that this value is a solution to system $H(w_{1},w_{2},G)$. 

These considerations also uniquely determine state $w_{2}$, given $w_{1}$: 
\[
 w_{2}[v]=\left\{\begin{array}{cc}              
              a-1, &\mbox{ if }v\in e_{k}\setminus e_{l}, \\
              -b,&\mbox{ if }v\in e_{l}\setminus e_{k}, \\
            w_{1}[v]-aw_{1}[v]-b(p-1-w_{1}[v]),&\mbox{ if }v\in e_{k}\cap e_{l}. \\
              \end{array}
       \right.
\]
This relation simply rewrites equality $w_{2}=w_{1}+a\cdot U_{k}+b\cdot U_{l}$. 

To conclude, we have to show that state $w_{2}$ is reachable from $w_{1}$ in graph $G$ restricted to edges $\{k,l\}$, where 
$w_{1}\neq {\bf 0},w_{2}\neq (-1;-1;\ldots;-1)$,

\[
 w_{1}[v]=\left\{\begin{array}{cc}              
             -1, &\mbox{ if }v\in e_{k}\setminus e_{l}, \\
              0,&\mbox{ if }v\in e_{l}\setminus e_{k}, \\
            \mbox{arbitrary},&\mbox{ if }v\in e_{k}\cap e_{l}. \\
              \end{array}
       \right.
\]

\[
 w_{2}[v]=\left\{\begin{array}{cc}              
              a-1, &\mbox{ if }v\in e_{k}\setminus e_{l}, \\
              -b,&\mbox{ if }v\in e_{l}\setminus e_{k}, \\
            w_{1}[v](b+1-a)+b,&\mbox{ if }v\in e_{k}\cap e_{l}. \\
              \end{array}
       \right.
\]

and $w_{1}+U_{k}\equiv {\bf 0}$, $w_{1}+U_{l}\equiv (-1;-1;\ldots;-1)$. 
\end{enumerate}

We will chose $z\in e_{k}\cap e_{l}$ and define state $W$ as follows: $W$ is specified by sum $w_{1}+a\cdot Z_{k}$, where $Z_{k}$ is a vector that coincides with $a\cdot U_{k}$ except at vertex $z$, where it is equal to $a\cdot U_{k}(z)-\lambda \mbox{(mod p)}$, with $\lambda\in {\bf Z}_{p}$ to be chosen later. 

We have chosen state $W$ in this particular way to allow us to apply the induction hypothesis $m=1$ and conclude that $W$ is reachable from $W_{1}$ (by making moves on $e_{k}$ only). We also want to argue using Case 1 that state $\Lambda$ is reachable from $W$ (by making moves on $e_{l}$ only), where $\Lambda$ is defined by vector $W+b\cdot U_{l}$. Finally, we want to use again the induction hypothesis with $m=1$ to argue that $w_{2}$ is reachable from $\Lambda$ (by making moves on $e_{k}$ only). 

To be able to accomplish all of these we need to satisfy (by the statement of Case 1) the following conditions: 
\begin{itemize} 
\item[(a)] $W\neq {\bf 0}$. 
\item[(b)] $\Lambda \neq (-1;-1;\ldots;-1)$. 
\item[(c)] $\Lambda \neq {\bf 0}$. 
\end{itemize} 

One can satisfy each condition by eliminating from consideration one possible value of $\lambda$ for each condition (a),(b),(c), and setting $\lambda$ to a remaining value 
that enforces (a),(b),(c) on vertex $z$. This already proves our claim in all cases but the one when $p=3$. In fact we can extend this argument to the case $p=3$ as well: condition  $W\equiv {\bf 0}$ uniquely identifies one value $\lambda_{0}\in {\bf Z}_{p}$. Then both choices $\lambda_{0}-1,\lambda_{0}+1$ lead to a state $W$ that satisfies (a). At least one of these two choices satisfies (b) and (c) as well. Indeed, the two resulting states $W_{\lambda=\lambda_{0}-1}$ and $W_{\lambda=\lambda_{0}+1}$ differ by $2\mbox{ (mod p)}$ at vertex $z$ and 
$-2 \mbox{ (mod p)}$ at other vertices $v$ of $e_{k}$. It is not 
possible then that $W_{\lambda=\lambda_{0}-1},W_{\lambda=\lambda_{0}+1} \in \{ {\bf 0}, (-1;-1;\ldots;-1)\}$. Hence at least one of the two choices satisfies all of (a), (b) and (c).   

\end{itemize} 
\end{proof}

\section{From Reachability to Recurrence} 

We have seen that reachability is easy to test. In the next result we show that recurrence essentially reduces to two 
 reachability tests: 

\begin{theorem}\label{recurrence} 
In conditions of Theorem~\ref{trees}, 
 given hypergraph $G=(E,V)$ and states $w_{1},w_{2}\in {\bf Z}_{p}^{n}$, $w_{1}\neq {\bf 0}$, state $w_{2}$ is a recurrent state for the dynamics started at 
$w_{1}$ if and only if: 
\begin{itemize} 
 \item[(1)] $w_{2}$ is reachable from $w_{1}$. 
\item[(2)] State ${\bf 0}$ is {\em not} reachable from $w_{1}$. 
\end{itemize}
\end{theorem}
\begin{proof} 
 Necessity of the two conditions is trivial. Suppose therefore that conditions (1) and (2) are satisfied, and let $w_{3}\in {\bf Z}_{p}^{n}$
be a state reachable from $w_{1}$. State $w_{3}\neq {\bf 0}$ because of condition (2). On the other hand let $Y_{1}$ be a solution of the system $H(G,w_{1},w_{3})$ 
and $Y_{2}$ be a solution of the system $H(G,w_{1},w_{2})$. One can immediately verify that $Y=Y_{2}-Y_{1}$ (where the difference 
is taken component-wise in ${\bf Z}_{p}$) is a solution of the system $H(G,w_{3},w_{2})$. Applying Theorem~\ref{trees} we infer that 
$w_{2}$ is reachable from $w_{3}$. 
\end{proof}

\begin{corollary} 
 Consider the Markov Chain specified by running the ${\bf Z}_{p}$-annihilating random walk on a good hypergraph $G$.
\begin{enumerate} 
 \item Transient states for the dynamics are those states $0\neq w\in {\bf Z}_{p}^{n}$ such that system $H(G,w,0)$ is solvable. 
\item All other states are either recurrent or inaccessible, depending on the starting point for the dynamics. 
\end{enumerate}
 
\end{corollary}

\section{Further Comments}

It would be interesting %(especially in light of Observation~\ref{counter}) 
to extend the results on reachability and recurrence to general hypergraphs. Clearly some changes have to be made to the final result; we believe, though, that a connection with linear algebra ultimately exists. 
  
The other issue for further study raised by this paper, more interesting in light of the connection with annihilating random walks) is the dynamics of {\em modular lights-out games under random update}, seen as finite state Markov chains (see \cite{aldous-fill-book} Chapter 14 and \cite{istrate-balance} for related results).
Recent related results  considers random lights-out games \cite{random-lights-out} and random Seidel switching \cite{random-seidel-switching} on graphs. It would be interesting to complete the analysis in this paper with one of the convergence time 
of the associated Markov chain.

Finally, not that the antivoter model was used in the analysis of a randomized algorithm for 2-coloring a graph 
\cite{donnelly-welsh-coloring}. 
This was later extended to colorings with more than two colors or other restrictions (e.g. \cite{petford1989randomised,mcdiarmid1993random}, see also \cite{frieze2007survey}) and 2-colorings 
of hypergraphs. Whether cyclic antivoter models and related concepts are useful for analyzing  randomized coloring algorithms is an interesting 
issue. 

\section{Acknowledgment}

%We thank anonymous referees that have discovered slight errors in a preliminary version of this paper. \
This work has been supported by CNCS IDEI Grant PN-II-ID-PCE-2011-3-0981 "Structure and computational difficulty in combinatorial optimization: an interdisciplinary approach".
%\bibliographystyle{alpha}
%\bibliography{/home/gistrate/Dropbox/texmf/bibtex/bib/bibtheory}

\end{document}

%% file: includes.tex
\usepackage{color}
\usepackage{textcomp}
\usepackage{amsmath}
\usepackage{amssymb}
\usepackage{graphicx}
\usepackage{latexsym}
\newcommand{\junk}[1]{}
\newtheorem{theorem}{Theorem}
\newtheorem{corollary}{Corollary}
\newtheorem{lemma}{Lemma}

\newtheorem{definition}{Definition}

\newtheorem{observation}{Observation}